\documentclass[11pt]{article}
\usepackage{amssymb}
\usepackage{amsmath}
\usepackage{amsthm}
\usepackage{verbatim}
\usepackage{pifont}
\usepackage{geometry}
\usepackage{fancyhdr}
\usepackage{enumerate}
\usepackage{graphicx}
\usepackage{multirow}
\usepackage[utf8]{inputenc}

\newtheorem{theorem}{Theorem}

\newtheorem{corollary}[theorem]{Corollary}

\theoremstyle{definition}
\newtheorem{remark}[theorem]{Remark}

\begin{document}
\title{New upper and lower bounds for the additive degree-Kirchhoff index}
\author{\emph{Monica Bianchi $^{1}$} \thanks{%
e-mail: monica.bianchi@unicatt.it} \and \emph{Alessandra Cornaro $^{1}$}
\thanks{%
e-mail: alessandra.cornaro@unicatt.it} \and \emph{Jos\'{e} Luis Palacios $%
^{2}$} \thanks{%
e-mail: jopala@usb.ve} \and \emph{Anna Torriero$^{1}$} \thanks{%
e-mail: anna.torriero@unicatt.it} }
\date{ }

\maketitle

\centerline{$^{{}^1}$ Department of Mathematics and Econometrics,  Catholic University, Milan,
Italy.}

\centerline{$^{{}^2}$ Department of Scientific Computing and Statistics,
Sim\'on Bol\'{\i}var University, Caracas, Venezuela.}

\begin{abstract}
Given a simple connected graph on $N$ vertices with size $|E|$ and degree sequence $d_{1}\leq d_{2}\leq ...\leq d_{N}$, the aim of this paper is to exhibit new upper and lower bounds for the
additive degree-Kirchhoff index in closed forms, not containing effective
resistances but a few invariants $(N,|E|$
and the degrees $d_{i}$)  and applicable in general contexts. In our arguments we follow a dual approach: along with a traditional toolbox of
inequalities we also use a relatively newer method in Mathematical Chemistry,
based on the majorization and Schur-convex functions. Some
theoretical and numerical examples are provided, comparing the bounds obtained here and those previously known in the literature.
\end{abstract}

\section{Introduction}

The Kirchhoff index $R(G)$ of a connected undirected graph $G=(V,E)$ with vertex set $\{1, 2, \ldots, N\}$ and edge set $E$ was defined by Klein and Randi\'c \cite{KleRan} as
$$R(G)=\sum_{i<j}R_{ij},$$
where $R_{ij}$ is the effective resistance of the edge $ij$.  A lot of attention has been given in recent years to this index, as well as to several modifications of it that take into account the degrees of the graph under consideration.  Indeed, Chen and Zhang defined in \cite{CheZha} the multiplicative degree-Kirchhoff index as
\begin{equation}
\label{multip}
R^*(G)=\sum_{i<j}d_id_jR_{ij},
\end{equation}
where $d_i$ is the degree (i.e., the number of neighbors) of the vertex $i$.  References  \cite{BCPT1}, \cite{Bozkurt2012}, \cite{Pal2011b} and \cite{Pal2011} deal with this index.
Also, Gutman et al. defined in \cite{Gutman} the additive degree-Kirchhoff index as
\begin{equation}
\label{addit}
R^+(G)=\sum_{i<j}(d_i+d_j)R_{ij},
\end{equation}
and worked on the identification of graphs with lowest such degree among unicyclic graphs.
The additive degree Kirchhoff index is motivated by the degree distance of a graph, and these two indices are equal in case the graph $G$ under consideration is a tree, as can be seen if the effective resistance $R_{ij}$ in equation (\ref{addit}) is replaced by the distance in the graph between $i$ and $j$.  See reference (\cite{Dobrynin}) for other details. \\
Recently, one of the authors of this paper showed in \cite{Pal2013}, using Markov chain theory,  that for any graph $G$
\begin{equation}\label{markov}
R^+(G)\ge 2(N-1)^2,
\end{equation}
and the lower bound is attained by the complete graph. Also, in \cite{Pal2013} it was shown that for any $G$
$$R^+(G)\le {1\over 3}(N^4-N^3-N^2+N),$$
and it was conjectured that the maximum of $R^+(G)$ over all graphs is attained by the $({1\over 3}, {1\over 3}, {1\over 3})$ barbell graph, which consists of two complete graphs on ${N\over 3}$ vertices united by a path of length ${N\over 3}$, and for which $R^+(G)\sim {2\over {27}}N^4$.
\vskip .2 in
The aim of the current article is to exhibit new upper and lower bounds for the additive degree-Kirchhoff index in closed forms,  not containing effective resistances but a few invariants ($N$,  $|E|$ and the degrees $d_i$), and applicable in general contexts. In what follows we only consider simple, undirected and connected graphs. In computing our bounds we follow a dual approach:
first we use a traditional toolbox of inequalities and then a relatively newer method in Mathematical Chemistry, based on majorization and Schur-convex
functions. Schur-convexity and majorization order are widely discussed in \cite{Marshall} and previous uses of the majorization partial order in
chemistry and a general overview are given in \cite{KleRan}.
One major advantage of this technique is to provide a unified
approach to recover many bounds in the literature as well as to obtain better ones. This technique has been applied in \cite{BT}, \cite{BCT1}, \cite{BCPT1}, \cite{BCPT2} for determining bounds of some relevant topological indicators of graphs which can be usefully expressed as Schur-convex functions.

\section{Lower bounds}

In order to produce new lower bounds for $R^+(G)$ we use  the following inequalities for the effective resistances that can be found in \cite{Pal2011b}:

\begin{equation}
\label{una}
R_{ij}\ge {{d_i+d_j-2}\over {d_id_j-1}},
\end{equation}
 in case $(i,j)\in E$ and
\begin{equation}
\label{dos}
R_{ij}\ge {1\over {d_i}}+{1\over {d_j}},
\end{equation}
 in case $(i,j)\notin E$.
 Then we can prove the following

 \begin{theorem}\label{th:first}
 For any graph $G$ with degree sequence $d_{1}\leq d_{2}\leq ...\leq d_{N}$,
 \begin{equation} \label{last}
 R^+(G)\ge N(N-4) + 2|E|\sum_{j=1}^N {1\over d_j}.
 \end{equation}
 \end{theorem}
 \begin{proof} Inserting (\ref{una}) and (\ref{dos}) into (\ref{addit}) we get
 \begin{equation}\label{eq:rifuno}
 \begin{split} R^+(G) &\ge \sum_{{i<j}\atop{d(i,j)=1}} {{(d_i+d_j)(d_i+d_j-2)}\over {d_id_j}}+\sum_{{i<j}\atop{d(i,j)>1}} (2+{d_i\over d_j}+{d_j\over d_i})\\
&=\sum_{{i<j}\atop{d(i,j)=1}} (2+{d_i\over d_j}+{d_j\over d_i})-2\sum_{{i<j}\atop{d(i,j)=1}} ({1\over d_i}+{1\over d_j})+\sum_{{i<j}\atop{d(i,j)>1}} (2+{d_i\over d_j}+{d_j\over d_i})\\
&=\sum_{i<j}2+\sum_{i<j}({d_i\over d_j}+{d_j\over d_i})-2\sum_{{i<j}\atop{d(i,j)=1}} ({1\over d_i}+{1\over d_j})\\
 &=N(N-1)+\sum_{i<j}({d_i\over d_j}+{d_j\over d_i})-2N.
 \end{split}
 \end{equation}
 After some algebra, it is not difficult to see that
$$\sum_{i<j}({d_i\over d_j}+{d_j\over d_i})=2|E|\sum_{j=1}^N {1\over d_j}-N,$$
and inserting into \eqref{eq:rifuno} we get the  bound (\ref{last}).
\end{proof}

The expression of the lower bound given in (\ref{last})  depends on the  summation $\sum_{j=1}^N {1\over d_j}$. Working on it, we get the following results

\begin{theorem}\label{th:segundo} Let $G$ be a graph with degree sequence $d_1 \le d_2 \le \cdots \le d_N$.
If $d_j=1$ for $1\le j \le M<N$ then
\begin{equation}\label{fino}
R^+(G) \ge N(N-4) +2|E|\left[M+{{(N-M)^2}\over {2|E|-M}}\right].
\end{equation}
\end{theorem}
\begin{proof}
If $d_j=1$ for $1\le j \le M<N$,  then
 \begin{equation}\label{extra}
R^+(G) \ge N(N-4)+2|E|M+ 2|E| \sum_{j=M+1}^N {1\over d_j}
\end{equation}
Applying  the harmonic mean-arithmetic mean inequality to the last summation in the above inequality, we obtain
$$\sum_{j=M+1}^N {1\over d_j}\ge {{(N-M)^2}\over {2|E|-M}},$$
and inserting this into (\ref{extra}) we get the desired result.
\end{proof}

 In the case of trees, by \eqref{fino} it is easy to obtain the following

\begin{corollary} If $T$ is a tree with $M\ge 2$ leaves and $N>2$ then
\begin{equation} \label{finoarbol}
R^+(T)\ge N(N-4)+2(N-1)\left[M+{{(N-M)^2}\over {2(N-1)-M}}\right].
\end{equation}
\end{corollary}

Now we will present three additional lower bounds for $R^+(G)$ starting again from (\ref{una}) and (\ref{dos}) and studying the behavior of a suitable real function depending on the degree sequence.  Later on we will discuss which bounds turn out to be better.

 \begin{theorem}\label{th:second}
 For any graph $G$  with degree sequence $d_1 \le d_2 \le \cdots \le d_N$, $N>2$,
 \begin{equation} \label{last1}
 R^+(G)\ge N(N-2) + 2|E|\sum_{j=1}^N {1\over d_j} - \dfrac {4|E|}{1+d_1}.
 \end{equation}
 \end{theorem}
 \begin{proof} Inserting (\ref{una}) and (\ref{dos}) into (\ref{addit}) we get
 \begin{equation} \label{bound}
 \begin{split}
 R^+(G) &\ge \sum_{{i<j}\atop{d(i,j)=1}} {{(d_i+d_j)(d_i+d_j-2)}\over {d_id_j-1}}+\sum_{{i<j}\atop{d(i,j)>1}} (2+{d_i\over d_j}+{d_j\over d_i})\\
&=\sum_{{i<j}\atop{d(i,j)=1}} \left[ \dfrac {(d_i+d_j)(d_i+d_j-2)}{d_id_j-1} - (2+{d_i\over d_j}+{d_j\over d_i}) \right ] + \sum_{i<j} (2+{d_i\over d_j}+{d_j\over d_i})\\
& = \sum_{{i<j}\atop{d(i,j)=1}} \left[ \dfrac {(d_i+d_j)(d_i+d_j-2d_id_j)}{d_id_j(d_id_j-1)} \right ] +N(N-2) + 2|E|\sum_{i=1}^n \frac 1 {d_i}.
 \end{split}
 \end{equation}
To bound the first term, let us consider the real function $f(x)=\frac {(x+d_j)(x+d_j-2d_jx)}{d_jx(d_jx-1)} $ in the interval $I=[d_1,d_N]$, for $d_j \ge 2$. By Calculus, this function is increasing for $x \ge 2$ and moreover $f(2) \ge f(1)$. Thus for any integer $x \in I$ we get $f(x) \ge \frac {(d_1+d_j)(d_1+d_j-2d_jd_1)}{d_jd_1(d_jd_1-1)} $. A similar argument applied to the function $g(x)= \frac {(d_1+x)(d_1+x-2d_1x)}{d_1x(d_1x-1)}$ in the interval $I'=[2,d_N]$ shows that $g$ is increasing in $I'$ and thus $g(x) \ge g(d_1) = - \frac {4}{1+d_1}$ if $d_1 \ge 2$. When $d_1=1$, for $x \ge 2$, we get $g(x)= - \frac {x+1} x$  and $g(x) \ge g(2)= - \frac 32>-2.$ Therefore
\begin{equation}\label{extra1}
\sum_{{i<j}\atop{d(i,j)=1}}  \dfrac {(d_i+d_j)(d_i+d_j-2d_id_j)}{d_id_j(d_id_j-1)} \ge - \dfrac {4|E|} {1+d_1}.
\end{equation}
Inserting (\ref{extra1}) in (\ref{bound}) we get (\ref{last1}).
\end{proof}

Finally, applying the harmonic mean-arithmetic mean inequality, from (\ref{last1}) we deduce the following corollaries
\begin{corollary}
For any graph $G$ with degree sequence $d_1 \le d_2 \le \cdots \le d_N$, $N>2$,
\begin{equation}\label{bound1}
R^+(G) \ge  2N(N-1)- \dfrac {4|E|} {1+d_1}.
\end{equation}
\end{corollary}

\begin{corollary}
If $d_j=1$ for $1\le j \le M<N$ then
\begin{equation}\label{fino1}
R^+(G) \ge N(N-2) +2|E|\left[M+{{(N-M)^2}\over {2|E|-M}}\right]-2|E|.
\end{equation}
\end{corollary}

\begin{corollary} If $T$ is a tree with $M\ge 2$ leaves then
\begin{equation} \label{finoarbol1}
R^+(T)\ge N(N-2)+2(N-1)\left[M+{{(N-M)^2}\over {2(N-1)-M}}\right]-2(N-1).
\end{equation}
\end{corollary}

\begin{remark}
 It is a simple exercise in Calculus to show that the real functions
$$\Phi(x)=x+{{(N-x)^2}\over {2|E|-x}},\,\,\,  x \ge 0$$
and
$$\Psi(x)=x\left[M+{{(N-M)^2}\over {2x-M}}\right],\,\,\,  x \ge N-1$$
are increasing.

The fact that $\Phi$ is increasing tells us that the bounds (\ref{fino}), (\ref{finoarbol}), (\ref{fino1}) and (\ref{finoarbol1})  improve as the number of vertices with degree 1 increases. Indeed,  the bound \eqref{fino} is  worse than the universal bound \eqref{markov} only when $M=0$,  a case we intentionally disregarded in the statement of Theorem \ref{th:segundo}, while for $M\ge 1$ our new bound betters \eqref{markov}.
In fact, since $\Phi(x)$ and $\Psi(x)$ are increasing, if $M\ge 1$, then either $M=1$ and $|E|$ must be at least $N$, because $G$ cannot be a tree with one leaf, or $M\ge 2$.
In the first case, $M=1$ and $|E|=N$, a simple computation yields that the bound (\ref{fino}) is $N^2-4N+2N\left[\dfrac{N^2}{2N-1}\right]$ which is better than \eqref{markov} whenever $N\ge 4$.
In the second case, with $M=2$ and $|E|=N-1$, an easy calculation yields that the bound (\ref{fino}), for $N>2$, is $2N^2-3N-2$, which is better than \eqref{markov} for $N\geq4$.
So by the monotonicity of $\Phi$ and $\Psi$ all cases are covered except possibly $M=1$ and $N=2$ or $3$. But it is impossible for a graph to have just one leaf and 2 or 3 vertices.

The fact that $\Psi$ is increasing tells us that (\ref{fino})  improves as the number of edges $|E|$ increases, so in a sense the lower bound (\ref{finoarbol}) is weakest.  This is noticeable when $M$ is small, as in the linear graph, where the bound is quadratic and the value of $R^+$ is cubic, but it is not so when $M$ is large, as in the $N$-star graph, where the bound (\ref{finoarbol})  becomes $3N^2-8N+4$ and the actual value of $R^+$ is $3N^2-7N+4$.  In the opposite direction, with $M$ small and $|E|$ large, \eqref{fino} performs well: if we take a complete $K_{N-1}$ with a single vertex attached with a single edge to anyone of the vertices of the $K_{N-1}$, then we get $R^+=3N^2-8N+8-{2\over {N-1}}$, whereas the lower bound is $3N^2-9N+6$.
\end{remark}

Summing up, by easy computations we have that:

\begin{enumerate}
\item  bound (\ref{finoarbol1}) is always better than bound (\ref{finoarbol});
\item  bound (\ref{fino}) is better than (\ref{fino1}) for $ |E|>N $ and they coincide for $|E|=N $, i.e. for the unicyclic graph;
\item  bound (\ref{bound1}) is better than (\ref{markov}) if and only if
\begin{equation}\label{condition}
(1+d_1)(N-1)>2|E|
\end{equation}

and they coincide, for example, in case of  trees and complete graphs. Note that (\ref{condition}) is satisfied, for instance, in case of $d$-regular graphs, with $d<N-1$.
\end{enumerate}

\section{Lower bounds via the majorization technique}
In this Section we show how  majorization can be applied to bound the additive degree-Kirchhoff index. This approach can be pursued if we can identify a set of variables with constant sum and a Schur-convex  function $f$ to be optimized on the set $S$ of these variables. In this case the global  minimum (maximum) of $f$ is attained at the minimum (maximum) element of the set $S$ with respect to the majorization order (see \cite{BT} and \cite{BCT1} for more details).

\begin{theorem}\label{th:major1}
For any graph $G$ with degree sequence $d_1\le d_2 \le \cdots \le d_n$, let
\begin{equation}\label{eq:k}
\underset{i<j}{\sum }\dfrac{d_j}{d_i}=H.
\end{equation}
Then
\begin{equation}\label{eq:new_bound}
R^+(G) \ge N(N-3) + H + \left [ \frac {N(N-1)} 2\right ]^2 \frac 1 {H}.
\end{equation}
\end{theorem}
\begin{proof}
Let us consider the $\dfrac{N(N-1)}{2}$ variables $x_{ij}=\dfrac{d_j}{d_i}$, with $i<j$.  The function
$$
f(x_{12}, x_{13}, \cdots x_{(N-1)N})= \sum_{i<j} \left ( \frac {d_i}{d_j} + \frac {d_j}{d_i} \right ) = \sum_{i=1}^{N-1} \sum_{j=i+1}^N  \left ( x_{ij} + \frac 1 {x_{ij}}  \right )
$$
is Schur-convex in the variables $x_{ij}$.

The minimal element of the  set
$$
\Sigma_{H}= \{ \mathbf{w} \in \mathbb{R}^{N(N-1)/2} : w_1 \ge w_2 \ge \cdots w_{N(N-1)/2} \ge 0 \,\, ,  \sum_i w_i =H \}
$$
with respect to the majorization order is $\frac {2H} {N(N-1)} \mathbf{s}^{N(N-1)/2}$, where $\mathbf{s}$ is the unit vector (see \cite{Marshall} ).  Thus we get the lower bound:
\begin{equation}\label{eq:first bound}
\sum_{i=1}^{N-1} \sum_{j=i+1}^N  \left ( x_{ij} + \frac 1 {x_{ij}}  \right ) \ge H + \left [ \frac {N(N-1)} 2\right ]^2 \frac 1 {H}.
\end{equation}
By the inequality (7) in Theorem \ref{th:first} we know that
$$
R^+(G) \ge N(N-3) + \sum_{i<j} \left ( \frac {d_i}{d_j} + \frac {d_j}{d_i} \right )
$$
and by \eqref{eq:first bound} we obtain the expected bound.
\end{proof}

\begin{remark}
The majorization technique would work just as fine if we considered the $\dfrac{N(N-1)}{2}$ variables $\dfrac 1 {x_{ij}} = \dfrac{d_i}{d_j}$ and the invariant quantity
\begin{equation}\label{eq:k**}
\underset{i<j}{\sum }\dfrac{d_i}{d_j}=H^{*}.
\end{equation}
Following the same steps as above we have another lower bound
\begin{equation}\label{eq:second  bound}
\sum_{i=1}^{N-1} \sum_{j=i+1}^N \left ( x_{ij} + \frac 1 {x_{ij}}  \right ) \ge H^{*} + \left [ \frac {N(N-1)} 2\right ]^2 \frac 1 {H^{*}}.
\end{equation}
Except for $d$-regular graphs for which the bounds \eqref{eq:first bound} and \eqref{eq:second  bound} coincide,  because $H=H^*$, in all other cases the first bound is always better than the second one.
In fact, by means of the harmonic mean - arithmetic mean inequality we get:
$$
H^{*} \cdot H \ge \left [ \frac {N(N-1)} 2 \right ]^2
$$
If $G$ is not a $d$-regular graph, the inequality $H > \frac {N(N-1)}2 > H^{*}$ holds.   Multiplying both sides of the last inequality by $(H-H^{\ast})$ yields:
\[
(H-H^{*}) \cdot H \cdot H^{*} \geq (H-H^{\ast }) \left[ \frac{N(N-1)}{2}\right]
^{2}.
\]
It follows that
$$
(H-H^{*} )\geq \left( \dfrac{H-H^*}{H \cdot H^{*}} \right) \left[ \dfrac{N(N-1)}{2}\right] ^{2}=\left( \dfrac{1}{H^{*}}-\dfrac{1}{H}\right) \left[ \dfrac{N(N-1)}{2}\right] ^{2}.
$$
Rearranging this inequality we conclude that  the lower bound in \eqref{eq:first bound} is always better than  the one in \eqref{eq:second  bound}.
\end{remark}

\noindent The usefulness of  \eqref{eq:new_bound} is limited by the computation of the graph invariant $H$.
We will list later  some examples of graphs for which this computation can be easily handled and compare  \eqref{eq:new_bound}  with the other bounds.
\vskip .2 in
Majorization is also the main argument in yet another possible approach for obtaining lower bounds.  In reference \cite{Pal2013} it was shown that the following relationship between the additive and multiplicative degree-Kirchhoff indices holds :
  \begin{equation}
 \label{rela}
 R^+(G)={N\over {2|E|}}R^*(G)+\sum_{j=1}^N\sum_{i\neq j} \pi_i E_iT_j,
 \end{equation}
 where $E_iT_j$ is the expected value of the number of steps $T_j$ that the random walk on $G$, started from vertex $i$, takes to reach vertex $j$.  We recall that this random walk moves from a vertex $v$ to any neighboring vertex $w$ with uniform probabilities $p(v,w)={1\over {d_v}}$ and that the $N\times N$ matrix $P=\left(p(v,w)\right)$ of transition probabilities has a unique probabilistic left eigenvector $\pi=(\pi_i)$ (the stationary distribution),  which is present in the summation in (\ref{rela}), and a spectrum $1=\lambda_1>\lambda_2\ge \lambda_3\ge \cdots \ge \lambda_N\ge -1$ in terms of which  $R^*(G)$ can be expressed (see \cite{Pal2011}), namely
 $$\displaystyle R^*(G)=2|E|\sum_{i=2}^N {1\over {1-\lambda_i}}.$$
With the preceding remarks and notation we can prove now the following

 \begin{theorem}
 For any graph $G$
 \begin{equation}\label{eq:bound_major1}
 R^+(G)\ge N\left[ {1\over {1+{\sigma \over \sqrt{N-1}}}}+{{(N-2)^2}\over {N-1-{\sigma \over \sqrt{N-1}}}}\right]+(N-1)^2,
 \end{equation}
 where
 $$\sigma^2={2\over N}\sum_{(i,j)\in E}{1\over {d_id_j}}=\frac {{\rm tr}(P^2)}N$$
 \end{theorem}
\begin{proof}
First, bound the summation with the hitting times as in \cite{Pal2013}:
 $$R^+(G)\ge N \sum_{i=2}^N {1\over {1-\lambda_i}} + 2|E|\sum_{j=1}^N{1\over {d_j}}-2N+1.$$
 Now apply the harmonic mean-arithmetic mean inequality only to the second addendum in order to get
 $$R^+(G)\ge N \sum_{i=2}^N {1\over {1-\lambda_i}} + (N-1)^2.$$
 Finally apply majorization to $\sum_{i=2}^N {1\over {1-\lambda_i}}$, as in \cite{BCPT1}, Proposition 11,  in order to get the expected bound \eqref{eq:bound_major1}.
\end{proof}

We remark that
 we recover the universal bound \eqref{markov} for the complete graph, for which $\sigma={1\over \sqrt{N-1}}$, and for all other graphs the bound is better than the universal one \eqref{markov}, as can be seen in the discussion in Section 3.1.1 of \cite{BCPT1}. \\
In order to complete our analysis, we study some particular classes of graphs for which the computation of $H$ is simple and we compare  \eqref{eq:new_bound}  with the other bounds.

\subsection{$d$-regular graph}
For $d$-regular graphs we have  $H=\dfrac{N(N-1)}{2}$. The lower bound \eqref{eq:new_bound} becomes $N(N-1)$ which is worse than bound \eqref{markov} and consequently worse than (\ref{bound1}) and \eqref{eq:bound_major1}.
%
\subsection{$(a,b)-$ semiregular graph}

Let us consider a semiregular graph that has $N_{1}$ vertices with degree $a$ and
$N_{2}$ vertices with degree $b$, $a<b$, $N=N_{1}+N_{2}$. Then $H=
\frac{N(N-1)}{2}+\left( \frac{b}{a}-1\right) N_{1}N_{2}$.\\
We deal with two examples: i) a semiregular bipartite graph and ii) a semiregular not bipartite graph.
\begin{itemize}
\item [i)] Let us consider a semiregular bipartite graph with $N_{1}$=10 vertices with degree $a=4$ and $N_{2}=4$ vertices with degree $b=10$.
For this graph we have that $H=151$, $\sigma=0.4689932$ which imply
\begin{table}[th]
\centering%
\begin{tabular}{|l|l|}
\hline
Bound (\ref{bound1}) &  $R^+(G)\ge 332$ \\ \hline
Bound (\ref{markov}) &  $R^+(G)\ge 338$ \\ \hline
Bound \eqref{eq:bound_major1} & $R^+(G)\ge 338.033$ \\ \hline
Bound \eqref{eq:new_bound} &  $R^+(G)\ge 359.64$ \\ \hline
\end{tabular}
\end{table}
\begin{center}
Table 1
\end{center}
\bigskip
\noindent Hence, bound \eqref{eq:new_bound} performs better than the others.
\item [ii)] Let us take a semiregular graph on $N$ vertices ($N$ even $\geq 8$) that is the union of a complete $K_{N/2}$ and a $N/2$-cycle such that vertex $i$ of the cycle is linked to vertex $i$ of the complete graph with a single edge, for $1\le i \le N/2$.

This graph has $N_{1}=N_{2}$ $=N/2
$, $a=3$, $b=N/2,$ thus $H=\frac{1}{24}N\left( 6N+N^{2}-12\right)$.
By \eqref{eq:new_bound} we get
\begin{equation}\label{bound_semireg1}\tag{19'}
R^+(G)\ge \frac{N\left( 228N^{2}-1152N+36N^{3}+N^{4}+1152\right) }{24\left(6N+N^{2}-12\right)}
\end{equation}

\noindent while bound (\ref{bound1}) becomes
\begin{equation}\label{bound_semireg2}\tag{14'}
R^+(G)\ge \frac{1}{8}N\left(15N-22\right)
\end{equation}

By Calculus, it is easy  to show that , for $N>8$, the bound (\ref{bound_semireg1}) is better than  both (\ref{bound_semireg2}) and (\ref{markov}). In virtue of (17), for $N>8$,  bound  (\ref{markov}) betters bound (\ref{bound_semireg2})
and they coincide for $N=8.$
\end{itemize}
We show in Table 2 a comparison between all bounds applicable to this example, for $N=20$:
\begin{table}[th]
\centering%
\begin{tabular}{|l|l|}
\hline
Bound (14') &  $R^+(G)\ge 695$ \\ \hline
Bound (\ref{markov}) &  $R^+(G)\ge 722$ \\ \hline
Bound (24) & $R^+(G)\ge 722.001$ \\ \hline
Bound (19') &  $R^+(G)\ge 848.61$ \\ \hline
\end{tabular}
\end{table}
\begin{center}
Table 2
\end{center}
Finally, looking at Table 2, the bound (19') performs always better than (24) which in turn improves both (3) and (14').
%

\subsection{Full binary tree of depth $d>1$}
We consider a full binary tree of depth $d>1$ which has $N_1=2^d$ vertices of degree 1, one vertex (the root) of degree 2 and $N_2=2^d-2$ vertices of degree 3. Then $H=\dfrac{N(N-1)}{2} +2N_1+\frac 3 2 N_2 + 2N_1N_2$.
Taking $d=3$ we obtain the results summarized in the following table, which shows that our new bounds are better than the universal one (3):
\begin{table}[th]
\centering
\begin{tabular}{|l|l|}
\hline
Bound (\ref{markov}) &  $R^+(G)\ge 392$ \\ \hline
Bound \eqref{eq:bound_major1} & $R^+(G)\ge 392.14$\\ \hline
Bound \eqref{eq:new_bound} &  $R^+(G)\ge   406$\\ \hline
Bound (14) &  $R^+(G)\ge  459.6$ \\ \hline
\end{tabular}
\end{table}
\begin{center}
Table 3
\end{center}
On the other hand, if we connect all the pendant vertices of the above mentioned full binary tree, we get a graph with $N_1=3$ vertices with degree $a=2$ and $N_2=12$ vertices with degree $d=3$.
Hence, we reduce it to the previous  class of $(2,3)-$ semiregular graph and we have \\
 \begin{table}[th]
\centering%
\begin{tabular}{|l|l|}
\hline
Bound (\ref{markov}) &  $R^+(G)\ge 392$ \\ \hline
Bound (\ref{bound1}) &  $R^+(G)\ge 392$ \\ \hline
Bound \eqref{eq:bound_major1} & $R^+(G)\ge  392.12$ \\ \hline
Bound \eqref{eq:new_bound} &  $R^+(G)\ge  392.63$ \\ \hline
\end{tabular}
\end{table}
\begin{center}
Table 4
\end{center}
In this case, bound \eqref{eq:new_bound} performs slightly better than the others.

%

%
%
%
%
%
%
%
%

\section{Upper bounds}

Now with respect to upper bounds, we can prove the following simple
\begin{theorem}
For any graph $G$ we have
\begin{equation}
\label{upper}
R^+(G)\le 2|E|(N-1)R,
\end{equation}
where $R=\max_{i,j} R_{ij}$.
\end{theorem}

\begin{proof}
$R^+(G)\le R\sum_{i<j}(d_i+d_j)=2|E|(N-1)R $.
\end{proof}

The inequality (\ref{upper}) may seem like a crude estimate. In fact, it recovers the right order of the upper bound, since
$$2|E|(N-1)R\le N(N-1)^3,$$
which is only worse that the bound found in \cite{Pal2013} by the constant of the largest $N^4$ term.  For trees, (\ref{upper}) becomes
\begin{equation}
\label{uppert}
R^+(T)\le 2(N-1)^2D\le 2(N-1)^3,
\end{equation}
where $D$ is the diameter of the graph.   Again, this inequality for trees gives the right order of the upper bound of the index, except at most for the constant of the $N^3$ term, as the linear graph shows,  since for this tree $R^+(G)\sim {2 \over 3}N^3$.  On the other hand, for the star graph, $D=2$ and (\ref{uppert}) becomes $4(N-1)^2$ which is off the actual value only by the constant $4$ instead of $3$.

In reference \cite{Mark} a careful study of $R$ is carried out for distance-regular graphs, a large family for which we can apply the simple bound (\ref{uppert}) and show that $R^+(G)$ must be less than twice the universal lower bound.

\begin{theorem}
If $G$ is distance-regular with degree $k>2$ then
$$R^+(G)\le \left(2+{{188}\over {101}}\right)(N-1)^2.$$
\end{theorem}

\begin{proof}
Insert the inequality $R\le\left(2+{{188}\over {101}}\right){{(N-1)}\over {Nk}}$ shown in \cite{Mark} (the equality holds only in the case of the Biggs-Smith graph) into (\ref{upper}) to prove the assertion.
\end{proof}

It is interesting to notice that the result is false for the distance-regular graphs with degree $k=2$, i.e., the cycles, for which $R^+(G)$ jumps from the quadratic values shown above for $k>2$ to the cubic value ${{N^3-N}\over 3}$.
\vskip .2 in
We conclude giving further upper bounds, in terms of the so-called spectral gap, obtained by combining ideas from Markov chains and majorization.

Recall that from (\ref{rela}) and subsequent comments we have
\begin{equation}
\label{rela2}
 R^+(G)={N\over {2|E|}}R^*(G)+\sum_{j=1}^N\sum_{i=1}^N \pi_i E_iT_j=N  \sum_{i=2}^N {1\over {1-\lambda_i}} + \sum_{j=1}^N\sum_{i=1}^N \pi_i E_iT_j
 \end{equation}
 We want to find an upper bound for the summation with the hitting times in (\ref{rela2}), for which we use  some Markov chain theory found in reference \cite{Lov}, specifically:
 \begin{equation}
 \label{lovasz}
 \sum_i \pi_iE_iT_j={1\over \pi_j}\sum_{k=2}^N{1\over {1-\lambda_k}}v_{kj}^2,
 \end{equation}
 where $v_{kj}$ is the $j$-th component of the eigenvector $v_k$ associated to the eigenvalue
 $\lambda_k$ (the vectors $v_k$ can be chosen to be orthonormal), and
 $$\sum_{k=2}^N v_{kj}^2=1-\pi_j.$$
It is clear that (\ref{lovasz}) can be bounded as follows:
$${1\over \pi_j}\sum_{k=2}^N{1\over {1-\lambda_k}}v_{kj}^2\le {1\over {(1-\lambda_2)\pi_j}}\sum_{k=2}^Nv_{kj}^2={1\over {1-\lambda_2}}{{1-\pi_j}\over \pi_j}.$$
 And so the sum of expected hitting times can be bounded as:
 $$\sum_{j=1}^N\sum_{i=1}^N \pi_i E_iT_j\le{1\over {1-\lambda_2}}\sum_j{{1-\pi_j}\over \pi_j}={1\over {1-\lambda_2}}(2|E|\sum_j{1\over d_j}-N).$$

  Now use in (\ref{rela2}) the upper bounds in \cite{BCPT1} Section 3.2 for $\sum_{i=2}^N {1\over {1-\lambda_i}}$,  to obtain the following corollaries:
\begin{corollary}
  For any $G$ we have
 \begin{equation}
 \label{final1}
 R^+(G)\le N\left({{N-k-2}\over {1-\lambda_2}}+{k\over 2}+{1\over \theta}\right)+{1\over {1-\lambda_2}}(2|E|\sum_j{1\over d_j}-N),
 \end{equation}
 where $\displaystyle k=\Bigg\lfloor {{\lambda_2(N-1)+1}\over {\lambda_2+1}}\Bigg\rfloor$ and $\displaystyle \theta=\lambda_2(N-k-2)-k+2$.
\end{corollary}

\begin{corollary}
For any bipartite $G$ we have
  \begin{equation}
 \label{final2}
 R^+(G)\le N\left({1\over 2}+{{N-k-3}\over {1-\lambda_2}}+{k\over 2}+{1\over \theta}\right)+{1\over {1-\lambda_2}}(2|E|\sum_j{1\over d_j}-N),
 \end{equation}
 where $k$ and $\theta$ are defined above.
\end{corollary}

For the $N$-star graph we have that $\lambda_2=0$, $k=1$ and $\theta=1$ and therefore the bound (\ref{final2}) becomes $3N^2-7N+4$ and the actual value of $R^+$ is attained.  This can be extended to the complete bipartite graph $K_{r,s}$, for arbitrary $r, s$, for which bound (\ref{final2}) becomes
\begin{equation}
\label{rs}
3r^2+3s^2+2rs-3r-3s,
\end{equation}
 whose order is always $N^2$, and improves the bound $2|E|(N-1)D=4rs(r+s-1)$.  The smallest value of (\ref{rs}) occurs for $r=s={N\over 2}$, where it takes the value $N(2N-3)$, which is equal to the actual value $N(2N-3)$ of $R^+(G)$.

%

\section{Conclusions}

We have derived upper and lower bounds for $R^+(G)$ whose expressions do not depend on the effective resistances, which in general are difficult to compute,  but on a limited number of graph invariants.  These bounds are not mutually exclusive and their performance depends on the particular structure of the graphs in question.  Here is a table summarizing our best results concerning the lower bounds:
\begin{table}[th!]
\centering
\begin{tabular}{|l|l|l|}
\hline
\textbf{Graph} & \textbf{Bound} & \\ \hline
\multirow{3}{*}{Generic} & $R^+(G) \ge  2N(N-1)- \dfrac {4|E|} {1+d_1}$ & (\ref{bound1})\\
    & $R^+(G) \ge N(N-3) + H + \left [ \frac {N(N-1)} 2\right ]^2 \frac 1 {H}$ & (19) \\
      &  $R^+(G)\ge N\left[ {1\over {1+{\sigma \over \sqrt{N-1}}}}+{{(N-2)^2}\over {N-1-{\sigma \over \sqrt{N-1}}}}\right]+(N-1)^2 $ & (24) \\ \hline
$M$ leaves & $R^+(G) \ge N(N-4) +2|E|\left[M+{{(N-M)^2}\over {2|E|-M}}\right]$ & (\ref{fino}) \\ \hline
Tree & $R^+(T)\ge N(N-2)+2(N-1)\left[M+{{(N-M)^2}\over {2(N-1)-M}}\right]-2(N-1)$ & (\ref{finoarbol1})\\ \hline
\end{tabular}
\end{table}
\begin{center}
Table 5
\end{center}
\newpage

\end{document}